\documentclass[12pt]{article}
\usepackage{amsfonts}
\usepackage[pctex32]{graphics}
\usepackage[dvips]{graphicx}
\usepackage{amsthm}
\usepackage{amsmath}
\usepackage{amssymb}
\usepackage[all]{xy}
\usepackage{latexsym}
\usepackage{layout}
\usepackage{epsfig}
\usepackage{graphicx}
\usepackage{pstricks,pst-node,pst-plot} 
\usepackage{color} 
\title{Omnimosaics}
\author {Katie R.~Banks\\
Department of Applied Mathematics\\
Harvard University\and
Anant P.~Godbole\\
Department of Mathematics and Statistics\\
East Tennessee State University\and
Nicholas George Triantafillou\\
Department of Mathematics\\
University of Michigan, Ann Arbor}
\begin{document}
\def\qed{\vbox{\hrule\hbox{\vrule\kern3pt\vbox{\kern6pt}\kern3pt\vrule}\hrule}}
\def\ms{\medskip}
\def\n{\noindent}
\def\ep{\varepsilon}
\def\G{\Gamma}
\def\tne{\lr{{t}\over{ne}}\rr^{2/t}}
\def\Gt{\Gamma(k+1)}
\def\Gn{\Gamma(n+1)}
\def\Gr{\Gamma(r+1)}
\def\Gc{\Gamma(c+1)}
\def\Gtr{\Gamma(k-r+1)}
\def\Gtc{\Gamma(k-c+1)}
\def\Gntr{\Gamma(n-k+r+1)}
\def\Gntc{\Gamma(n-k+c+1)}
\def\lr{\left(}
\def\lf{\lfloor}
\def\rf{\rfloor}
\def\lc{\left\{}
\def\rc{\right\}}
\def\rr{\right)}
\def\ph{\varphi}
\def\p{\mathbb P}
\def\v{\mathbb V}
\def\nk{n \choose k}
\def\a{\cal A}
\def\e{\mathbb E}
\def\l{\mathbb L}
\def\lg{{\rm lg}}
\newcommand{\bsno}{\bigskip\noindent}
\newcommand{\msno}{\medskip\noindent}
\newcommand{\oM}{M}
\newcommand{\omni}{\omega(k,a)}
\newtheorem{thm}{Theorem}[section]
\newtheorem{con}{Conjecture}[section]
\newtheorem{deff}[thm]{Definition}
\newtheorem{lem}[thm]{Lemma}
\newtheorem{cor}[thm]{Corollary}
\newtheorem{rem}[thm]{Remark}
\newtheorem{prp}[thm]{Proposition}
\newtheorem{ex}[thm]{Example}
\newtheorem{eq}[thm]{equation}
\newtheorem{que}{Problem}[section]
\newtheorem{ques}[thm]{Question}
\providecommand{\floor}[1]{\left\lfloor#1\right\rfloor}
\maketitle
\begin{abstract}  An {\it omnimosaic} $O(n,k,a)$ is defined to be an $n\times n$ matrix, with entries from the set ${\cal A}=\{1,2,\ldots,a\}$, that contains, as a submatrix, each of the $a^{k^2}$ $k\times k$ matrices over ${\cal A}$.  We provide constructions of omnimosaics and show that for fixed $a$ the smallest possible size $\omega(k,a)$ of an $O(n,k,a)$ omnimosaic satisfies
\[\frac{ka^{k/2}}{e}\le \omega(k,a)\le \frac{ka^{k/2}}{e}(1+o(1))\] for a well-specified function $o(1)$ that tends to zero as $k\to\infty$.
\end{abstract}
\section{Introduction}  We start with a discussion of the Bible Code.  The paper \cite{witzum} of Witztum, Rips and Rosenberg made the extraordinary claim that the Hebrew text of the Book of Genesis encoded events which did not occur until millenia after the text was written. Of particular note was that the predictions occurred in equidistant letter schemes.  In the rejoinder paper of McKay et al \cite{mckay} the authors exhibited that the claims of Witzum et al. were ``fatally defective," and  ``indeed that their result merely reflects on the choices made in designing their experiment and collecting the data for it."  More recently, Abraham et al. \cite{omni} defined {\it $k$-omnibus sequences} (or $k$-omni sequences, for short) as being $n$-long sequences over the alphabet ${\cal A}=\{1,2,\ldots,a\}$ that contained {\it each} sequence of length $k$ over $\a$ as a (not necessarily equidistant or contiguous) subsequence.  A computer check showed, for example, that an English translation of Tolstoy's {\it War and Peace} was 950-omni but not 951-omni.  An easy connection was made with the coupon collection problem, see e.g., \cite{feller}, \cite{zeil}:  A sequence of $n$ letters is $k$-omni if and only if it contains at least $k$ disjoint ``coupon collections".  This led to a threshold result for a random sequence being $k$-omni:
\begin{thm}
Let $r>0$ be a constant, and fix $a\geq 2$, $n=rk$, where $n,k$ are both integers.  Set $H(1..a)=1+\frac{1}{2}+\ldots\frac{1}{a}$.  Then
$$\lim_{k \rightarrow\infty} {\mathbb P}({\rm Sequence\ is\ {\it k}-omni}) = \left\{
\begin{array}{ll}
0, & \text{ if $r<a H(1..a)$, or}\\
1, & \text{ if $r>a H(1..a)$}
\end{array}\right.$$
\end{thm}
Furthermore, there is a different threshold for the expected number $\e(X)$ of missing $k$-subsequences to transition from asymptotically infinite to asymptotically zero; for $a=2$ for example, the sequence goes from being non-omni w.h.p. to being omni w.h.p. at $n=3k$, but $\e(X)\to\infty\enspace(n/k\le4.403)$ and $\e(X)\to0\enspace(n/k\ge4.403)$.

After a Johns Hopkins University colloquium talk given by AG, Dan Naiman asked if these results could be generalized to higher dimensions. Now steganography ({\tt http://en.wikipedia.org/wiki/Steganography}) has been used since ancient times, and recently by groups such as Al Qaeda and the alleged Russian spies in the U.S, but how large would a random  image have to be, Naiman asked, before it contains huge numbers of images of a smaller size?  This paper is an attempt to provide some answers.  We retain the prefix adjective {\it omni} and combine it with the phrase {\it mosaic} to describe a landscape so rich that it contains all ``color pictures" of a smaller size, including gibberish, familiar photographs such as that of the 1969 moon landing, and photos of events to occur many millenia into the future:  An {\it omnimosaic} $O(n,k,a)$ is defined to be an $n\times n$ matrix, with entries from the set ${\cal A}=\{1,2,\ldots,a\}$, that contains, as a submatrix, each of the $a^{k^2}$ $k\times k$ matrices over $\a$.  When $k,a$ are fixed, the smallest $n$ for which an $O(n,k,a)$ omnimosaic exists is denoted by $\omega(k,a)$; the example
$$\begin{pmatrix}0&1&0&1\cr
1&0&1&0\cr
0&1&0&0\cr
0&1&1&1\end{pmatrix}
$$ shows that $\omega(2,2)=4$.  If $a=2$ we quickly see that an omnimosaic is a bipartite graph with $n$ elements in each color class so that each possible $k\times k$ bipartite graph occurs as an induced subgraph (with isomorphic graphs counting as separate cases).  In this respect our work continues along the lines of the vast body of work on {\it Universal Graphs} done by Moon \cite{moon};  Chung and her colleagues \cite{chung1}, \cite{chung2}; and more recent authors such as in \cite{alstrup}, \cite{butler}, and \cite{frieze}. A good historical account that includes many more references can be found in \cite{alstrup}.  The literature seems thus far to have focused on graphs that are induced universal for all graph isomorphisms; graphs that are universal or induced universal for families of graphs; and random graphs. In the nomenclature of the above authors, when $a=2$, omnimosaics would likely  be termed as ``bipartite induced universal graphs."

In Section 2 we prove that $\omega(2,3)=6$ and provide explicit omnimosaic constructions for all $k,a$, giving us an upper bound for the minimal size of an omnimosaic.  Then, in Section 3, we use Suen's correlation inequality \cite{janson} to show that
\begin{equation}\frac{ka^{k/2}}{e}\le \omega(k,a)\le \frac{ka^{k/2}}{e}(1+o(1))\end{equation} for a well-specified function $o(1)$ that tends to zero as $k\to\infty$ and which may be taken to be $2\log k/k$.  The lower bound in (1) is trivial and we dispose it off right away. There are ${{n}\choose{k}}^2$ $k\times k$ sub-matrices of an $n\times n$ matrix; these are to cover all $a^{k^2}$ possibilities so, by the pigeonhole principle we must have
\begin{equation}{{n}\choose{k}}^2\ge a^{k^2}.\end{equation}  Now, since, by a na\"{i}ve application of Stirling's formula 
\[{n\choose k}\le\lr\frac{ne}{k}\rr^k,\]
we see that we must have
\[n\ge\frac{ka^{k/2}}{e}\]
in order for the array to form an omnimosaic.

Now, had we been dealing with adjacency matrices of graphs rather than bipartite graphs ($a=2$), the above reasoning would have been replaced by
\[{n\choose k}\ge 2^{k\choose 2},\]
where $n$ is the number of vertices in the induced universal graph for graphs on $k$ vertices, which would have led to
\[
  n\ge \frac{\sqrt 2}{e}k2^{k/2}:=n_0\]
-- which is, upto a $(1+o(1))$ factor, precisely the best known lower bound on the diagonal Ramsey numbers (see, e.g., \cite{alon}).  Now, if we consider the random graph $G(n,1/2)$ and denote by $Z$ and $J$ the numbers of empty and complete graphs on $k$ vertices, the Ramsey bound tells us that
\[n\le n_0\Rightarrow \p(Z+J=0)>0.\]
The graph analog of our main result, states, on the other hand that
\[n\ge n_0(1+o(1))\Rightarrow \p({\rm all\ graphs\ are\ present})\ \to 1.\]  {\it Question.} Does this provides heuristic evidence that the {correct} asymptotic value of the diagonal Ramsey numbers is $b^k$ with $b$ closer to $\sqrt 2$ than to 4?
\section{Exact Value of $\omega(2,3)$ and a Constructive Upper Bound on Minimum Omnimosaic Size}  The pigeonhole bound of Equation (2) shows that for $k=2$ and $a=3$, we must have ${n\choose 2}\ge9$, or $n\ge5$.  We first prove that $n=5$ is impossible and then provide a general construction that will produce a $O(6,2,3)$ omnimosaic, showing that $\omega(2,3)=6$.  

Suppose that there exists a $5 \times 5$ omnimosaic with $a = 3$ and $k = 2$.  First, we will show that each row contains each letter of our alphabet.  Suppose the contrary.  Let the alphabet be $\{a,b,c\}$.  Without loss of generality, let $c$ be the missing letter in some row.  Then, all $81 - 16 = 65$ of the $2 \times 2$ matrices that contain a `$c$' must be submatrices of the $4 \times 5$ matrix that omits the given row.  However, there are only $\binom{4}{2}\binom{5}{2} = 60$ $2 \times 2$ submatrices in a $4 \times 5$ matrix, so this is impossible, yielding the desired contradiction.  Thus, each row contains each letter of our alphabet.  Notably, this means that each row either contains three of one letter and one each of the other two letters, or two of two letters and one of the remaining letter.

Now, we will focus on submatrices which contain the same letter in both positions in the bottom row.  For each letter, there are $9$ such $2\times 2$ matrices that must be included.  Now, we will count the number of submatrices of our omnimosaic that have a single letter in the bottom row.  Then, if the $i$th row contains three copies of one letter, say `$a$', and one copy each of `$b$' and `$c$', then there are $3(i-1)$ submatrices with bottom row `$aa$' from the $i$th row, and $0$ submatrices with bottom row `$bb$' and `$cc$' from the $i$th row.  Alternately, if the $i$th row contains one copy of one letter, say `$a$', and two copies each of `$b$' and `$c$', then there are $0$ submatrices with bottom row `$aa$' from the $i$th row, and $(i-1)$ submatrices with bottom row `$bb$' from the $i$th row and $(i-1)$ with bottom row `$cc$' from the $i$th row.  A review of the cases now reveals that for all $9$ of the desired $2\times 2$ matrices to appear for each letter, the second and third rows must contain three copies each of the same letter, the fourth row must contain three copies of a different letter, and the fifth row must contain three copies of the third letter.  However, this argument also holds if we enumerate the rows from the bottom rather than the top.  Then, the third and fourth rows must contain three copies of the same letter.  However, this is clearly a contradiction, so there is no $5 \times 5$ omnimosaic with $a = 3$ and $k = 2$.

We now present a construction for a {\it square} omnimosaic with side length $\left\lceil \frac{k}{2}\right\rceil a^{\left\lceil \frac{k}{2}\right\rceil} + \left\lfloor \frac{k}{2}\right\rfloor a^{\left\lfloor \frac{k}{2}\right\rfloor}$ along with a simple procedure for finding the target matrix in any construction made using this procedure. The main result of the section is as follows

\begin{thm} \label{Construction Upper Bound}
For any $a,k \in \mathbb N$, $\omni \leq \left\lceil \frac{k}{2}\right\rceil a^{\left\lceil \frac{k}{2}\right\rceil} + \left\lfloor \frac{k}{2}\right\rfloor a^{\left\lfloor \frac{k}{2}\right\rfloor}$
\end{thm}

Note that for any $a \in \mathbb N$, and $k$ even, Theorem 2.1 yields $\omni \leq ka^{\frac{k}{2}}$, and that for $k=2,a=3$ we get $\omega(2,3)=6$.

The construction presented here will be created based on a diagram of a $k \times k$ grid with each square containing either a horizontal or vertical line.  For a given $k \times k$ grid with either a horizontal or vertical line in each square, we define $r_i$ to be the number of horizontal lines in the $i$th row, and $c_j$ to be the number of vertical lines in the $j$th column.  In general, we will refer to the row and column numbers as if our $k \times k$ grid were a $k \times k$ matrix.  The following lemma will guarantee that an appropriate diagram exists to give us the desired result.

\begin{lem}\label{grid existence}
For any $k \in \mathbb N$, there exists a way of placing either a horizontal or vertical line in each square of a $k \times k$ grid such that $\left\lfloor \frac{k}{2} \right\rfloor$ of the $r_i$ and $\left\lceil\frac{k}{2}\right\rceil$ of the $c_j$ are equal to $\left\lceil\frac{k}{2}\right\rceil$, while $\left\lceil\frac{k}{2}\right\rceil$ of the $r_i$ and $\left\lfloor \frac{k}{2} \right\rfloor$ of the $c_j$ are equal to $\left\lfloor \frac{k}{2} \right\rfloor$.
\end{lem}

We will demonstrate this by a simple construction, establishing the following slightly stronger lemma.

\begin{lem}
For any $k \in \mathbb N$, there exists a way of placing either a horizontal or vertical line in each square of a $k \times k$ grid such that the following hold.
\begin{align*}
r_{1} = r_{2} = \cdots = r_{\left\lfloor \frac{k}{2} \right\rfloor} = c_{\left\lfloor \frac{k}{2} \right\rfloor + 1} = c_{\left\lfloor \frac{k}{2} \right\rfloor + 2} = \cdots = c_{k} &= \left\lceil\frac{k}{2}\right\rceil \\
c_{1} = c_{2} = \cdots = c_{\left\lfloor \frac{k}{2} \right\rfloor} = r_{\left\lfloor \frac{k}{2} \right\rfloor + 1} = r_{\left\lfloor \frac{k}{2} \right\rfloor + 2} = \cdots = r_{k} &= \left\lfloor\frac{k}{2}\right\rfloor
\end{align*}
\end{lem}

\begin{proof}
This lemma can be accomplished quite simply.  Consider the square in the $i$th row and the $j$th column.  If $i \leq \left\lfloor \frac{k}{2} \right\rfloor$ and $j \leq \left\lceil \frac{k}{2} \right\rceil$, place a horizontal line in this square.  Also, if $i > \left\lfloor \frac{k}{2} \right\rfloor$ and $j > \left\lceil \frac{k}{2} \right\rceil$, place a horizontal line in this square.  Otherwise, place a vertical line in this square.  Since $\left\lceil \frac{k}{2}\right\rceil + \left\lfloor \frac{k}{2} \right\rfloor = k$, it is clear that the desired property holds.  The diagrams on the left side of figures \ref{k=2} and \ref{k=3} demonstrate the design of this construction for $k = 2$ and $k = 3$ respectively.  The construction is visually similar for larger $k$.
\end{proof}

Now, we are ready to construct our omnimosaic.  The following construction and lemma will detail how to construct an omnimosaic based on a given $k \times k$ grid.  The main result of the section will be an immediate consequence.

Let $\mathcal K$ be a placement of horizontal and vertical lines in a $k \times k$ grid.  We will construct a $\left(\sum_{i = 1}^{k}a^{r_i}\right)\times \left(\sum_{i = 1}^{k}a^{c_i}\right)$ matrix $\mathcal M_{\mathcal K}$ based on this diagram.

First, consider an empty $\left(\sum_{i = 1}^{k}a^{r_i}\right)\times \left(\sum_{i = 1}^{k}a^{c_i}\right)$ matrix.  Let $s_l$ be the $l$th row of this matrix and $d_m$ be the $m$th column. We define the $i$th row-region $R_i$ to be the set $\left\{s_l : \sum_{t = 1}^{i-1}a^{r_t} < l \leq \sum_{t = 1}^{i}a^{r_t}\right\}$. Similarly, we define the $j$th column-region $C_j$ to be the set $\left\{d_m : \sum_{t = 1}^{j-1}a^{c_t} < m \leq \sum_{t = 1}^{j}a^{c_t}\right\}$. Then, the $(i,j)$th region of the matrix, denoted $\mathcal M'_{i,j}$ is the submatrix consisting of all elements in both the $i$th row-region and the $j$th column region.

Let $f_i$ be a bijection from the set $\{1,2,\ldots,r_i\}$ to the set of $j$ such that the square in the $i$th row and $j$th column of $\mathcal K$ contains a horizontal line.  Finally, let $W_{i} = W_{a,r_{i}}$ be the set of all $r_{i}$ letter words on the alphabet $a$, let $g_i = g_{a,r_{i}}: \{1,2,\ldots,a^{r_{i}}\} \to W$ be a bijection, and let $h(w,t)$ denote the $t$th letter of the word $w$.  

Then, for each $t \in \{1,2,\ldots,r_{1}\}$, let every element in the $l$th row of the $(1,f_i(t))$th region be $h(g_i(l),t)$.  Define similar functions for the vertical lines based on columns and repeat to fill in the remaining regions.  

This completes the construction of the matrix $\mathcal M_{\mathcal K}$.  In some sense, this construction can be seen as expanding the horizontal lines from the diagram so that by choosing a row in a given row-region, all elements corresponding to horizontal lines in the associated row of $\mathcal K$ are fixed, regardless of the choice of columns, with the expansion necessary to ensure that all possible sequences of letters are included.  This will be made formal in Lemma \ref{Construction Works}.  Figures \ref{k=2} and \ref{k=3} give examples of the process described here.

\begin{lem} \label{Construction Works}
$\mathcal M_{\mathcal K}$, as defined above, is an omnimosaic for $a,k$.
\end{lem}

\begin{proof}
Let $T$ be a target $k \times k$ matrix, let $T_i$ be its $i$th row, and let $T_{i,j}$ be the element in the $i$th row and $j$th column.  Then for each $i$, choose row $i$ to be the $g^{-1}(T_{i,f_i(1)}T_{i,f_i(1)}\cdots T_{i,f_i(r_i)})$th row of the $i$th row-region of $\mathcal M_{\mathcal K}$.  Then, all elements corresponding to a horizontal line in $\mathcal K$ will be correct, so long as the $j$th column selected belongs to the $j$th column-region of $\mathcal M_{\mathcal K}$.  Repeating this process on the columns guarantees this and also guarantees that all elements corresponding to a vertical line will match the associated element in $T$.  Thus, $T$ is a submatrix of $\mathcal M_{\mathcal K}$ for all $k \times k$ matrices $T$, so $\mathcal M_{\mathcal K}$ is an omnimosaic, as desired.
\end{proof}

\begin{lem} \label{omnimosaic Size}
For each placement of horizontal and vertical lines in a $k \times k$ grid, $\mathcal K$, there exists an omnimosaic with $\sum_{i = 1}^{k}a^{r_i}$ rows and $\sum_{i = 1}^{k}a^{c_i}$ columns.
\end{lem}

This lemma follows trivially from our construction and Lemma \ref{Construction Works}.

Now, we must simply apply Lemma \ref{omnimosaic Size} with $\mathcal K$ given by Lemma \ref{grid existence} to prove Theorem \ref{Construction Upper Bound}, since we can add additional rows to the resulting  omnimosaic, yielding a square omnimosaic with the desired side length.

\begin{figure}
\begin{center}
\begin{pspicture}(-2,-1)(6,1)

\pcline[linewidth=0.5pt,linecolor=black](-1,-1)(-1,1)
\pcline[linewidth=0.5pt,linecolor=black](0,-1)(0,1)
\pcline[linewidth=0.5pt,linecolor=black](1,-1)(1,1)
\pcline[linewidth=0.5pt,linecolor=black](-1,-1)(1,-1)
\pcline[linewidth=0.5pt,linecolor=black](-1,0)(1,0)
\pcline[linewidth=0.5pt,linecolor=black](-1,1)(1,1)

\pcline[linewidth=0.5pt,linecolor=black](-0.75,0.5)(-0.25,0.5)
\pcline[linewidth=0.5pt,linecolor=black](0.25,-0.5)(0.75,-0.5)
\pcline[linewidth=0.5pt,linecolor=black](0.5,0.25)(0.5,0.75)
\pcline[linewidth=0.5pt,linecolor=black](-0.5,-0.75)(-0.5,-0.25)

\pcline[linewidth=0.5 pt, linecolor=black,arrowsize=3pt 3]{->}(1.25,0)(2.75,0)

\pcline[linewidth=0.5pt,linecolor=black](3,-1)(3,1)
\pcline[linewidth=0.5pt,linecolor=black](4,-1)(4,1)
\pcline[linewidth=0.5pt,linecolor=black](5,-1)(5,1)
\pcline[linewidth=0.5pt,linecolor=black](3,-1)(5,-1)
\pcline[linewidth=0.5pt,linecolor=black](3,0)(5,0)
\pcline[linewidth=0.5pt,linecolor=black](3,1)(5,1)

\uput[90](3.33,0.0){$1$}
\uput[90](3.67,0.0){$1$}
\uput[90](3.33,0.45){$0$}
\uput[90](3.67,0.45){$0$}
\uput[90](4.33,0.0){$0$}
\uput[90](4.67,0.0){$1$}
\uput[90](4.33,0.45){$0$}
\uput[90](4.67,0.45){$1$}
\uput[90](4.33,-1.0){$1$}
\uput[90](4.67,-1.0){$1$}
\uput[90](4.33,-0.55){$0$}
\uput[90](4.67,-0.55){$0$}
\uput[90](3.33,-1.0){$0$}
\uput[90](3.67,-1.0){$1$}
\uput[90](3.33,-0.55){$0$}
\uput[90](3.67,-0.55){$1$}

\end{pspicture}
\end{center}
\caption{Base Diagram and Construction for $a = 2$, $k = 2$ with $n = 4$}
\label{k=2} 
\end{figure}
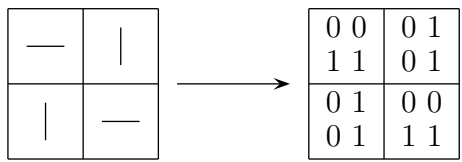

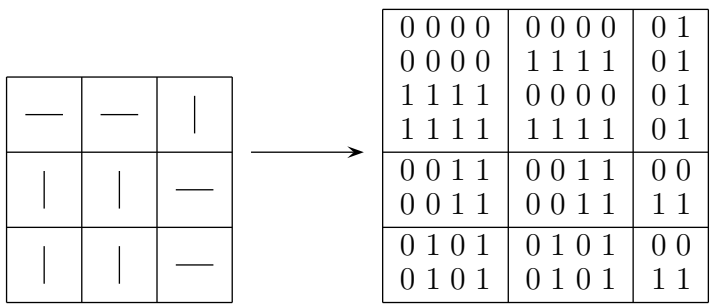
\begin{figure} 
\begin{center}
\begin{pspicture}(-3,-1)(9,3)

\pcline[linewidth=0.5pt,linecolor=black](-2,-1)(-2,2)
\pcline[linewidth=0.5pt,linecolor=black](-1,-1)(-1,2)
\pcline[linewidth=0.5pt,linecolor=black](0,-1)(0,2)
\pcline[linewidth=0.5pt,linecolor=black](1,-1)(1,2)
\pcline[linewidth=0.5pt,linecolor=black](-2,-1)(1,-1)
\pcline[linewidth=0.5pt,linecolor=black](-2,0)(1,0)
\pcline[linewidth=0.5pt,linecolor=black](-2,1)(1,1)
\pcline[linewidth=0.5pt,linecolor=black](-2,2)(1,2)

\pcline[linewidth=0.5pt,linecolor=black](-1.75,1.5)(-1.25,1.5)
\pcline[linewidth=0.5pt,linecolor=black](-0.75,1.5)(-0.25,1.5)
\pcline[linewidth=0.5pt,linecolor=black](0.5,1.75)(0.5,1.25)
\pcline[linewidth=0.5pt,linecolor=black](-1.5,0.25)(-1.5,0.75)
\pcline[linewidth=0.5pt,linecolor=black](-1.5,-0.75)(-1.5,-0.25)
\pcline[linewidth=0.5pt,linecolor=black](-0.5,0.25)(-0.5,0.75)
\pcline[linewidth=0.5pt,linecolor=black](-0.5,-0.75)(-0.5,-0.25)
\pcline[linewidth=0.5pt,linecolor=black](0.25,0.5)(0.75,0.5)
\pcline[linewidth=0.5pt,linecolor=black](0.25,-0.5)(0.75,-0.5)

\pcline[linewidth=0.5 pt, linecolor=black,arrowsize=3pt 3]{->}(1.25,1.0)(2.75,1.0)

\pcline[linewidth=0.5pt,linecolor=black](3,-1)(3,2.9)
\pcline[linewidth=0.5pt,linecolor=black](4.67,-1)(4.67,2.9)
\pcline[linewidth=0.5pt,linecolor=black](6.33,-1)(6.33,2.9)
\pcline[linewidth=0.5pt,linecolor=black](7.33,-1)(7.33,2.9)
\pcline[linewidth=0.5pt,linecolor=black](3,-1)(7.33,-1)
\pcline[linewidth=0.5pt,linecolor=black](3,0)(7.33,0)
\pcline[linewidth=0.5pt,linecolor=black](3,1)(7.33,1)
\pcline[linewidth=0.5pt,linecolor=black](3,2.9)(7.33,2.9)

\uput[90](3.33,2.35){$0$}
\uput[90](3.67,2.35){$0$}
\uput[90](4.00,2.35){$0$}
\uput[90](4.33,2.35){$0$}
\uput[90](5.00,2.35){$0$}
\uput[90](5.33,2.35){$0$}
\uput[90](5.67,2.35){$0$}
\uput[90](6.00,2.35){$0$}
\uput[90](6.67,2.35){$0$}
\uput[90](7.00,2.35){$1$}

\uput[90](3.33,1.9){$0$}
\uput[90](3.67,1.9){$0$}
\uput[90](4.00,1.9){$0$}
\uput[90](4.33,1.9){$0$}
\uput[90](5.00,1.9){$1$}
\uput[90](5.33,1.9){$1$}
\uput[90](5.67,1.9){$1$}
\uput[90](6.00,1.9){$1$}
\uput[90](6.67,1.9){$0$}
\uput[90](7.00,1.9){$1$}

\uput[90](3.33,1.45){$1$}
\uput[90](3.67,1.45){$1$}
\uput[90](4.00,1.45){$1$}
\uput[90](4.33,1.45){$1$}
\uput[90](5.00,1.45){$0$}
\uput[90](5.33,1.45){$0$}
\uput[90](5.67,1.45){$0$}
\uput[90](6.00,1.45){$0$}
\uput[90](6.67,1.45){$0$}
\uput[90](7.00,1.45){$1$}

\uput[90](3.33,1.0){$1$}
\uput[90](3.67,1.0){$1$}
\uput[90](4.00,1.0){$1$}
\uput[90](4.33,1.0){$1$}
\uput[90](5.00,1.0){$1$}
\uput[90](5.33,1.0){$1$}
\uput[90](5.67,1.0){$1$}
\uput[90](6.00,1.0){$1$}
\uput[90](6.67,1.0){$0$}
\uput[90](7.00,1.0){$1$}

\uput[90](3.33,0.45){$0$}
\uput[90](3.67,0.45){$0$}
\uput[90](4.00,0.45){$1$}
\uput[90](4.33,0.45){$1$}
\uput[90](5.00,0.45){$0$}
\uput[90](5.33,0.45){$0$}
\uput[90](5.67,0.45){$1$}
\uput[90](6.00,0.45){$1$}
\uput[90](6.67,0.45){$0$}
\uput[90](7.00,0.45){$0$}

\uput[90](3.33,0.0){$0$}
\uput[90](3.67,0.0){$0$}
\uput[90](4.00,0.0){$1$}
\uput[90](4.33,0.0){$1$}
\uput[90](5.00,0.0){$0$}
\uput[90](5.33,0.0){$0$}
\uput[90](5.67,0.0){$1$}
\uput[90](6.00,0.0){$1$}
\uput[90](6.67,0.0){$1$}
\uput[90](7.00,0.0){$1$}

\uput[90](3.33,-0.55){$0$}
\uput[90](3.67,-0.55){$1$}
\uput[90](4.00,-0.55){$0$}
\uput[90](4.33,-0.55){$1$}
\uput[90](5.00,-0.55){$0$}
\uput[90](5.33,-0.55){$1$}
\uput[90](5.67,-0.55){$0$}
\uput[90](6.00,-0.55){$1$}
\uput[90](6.67,-0.55){$0$}
\uput[90](7.00,-0.55){$0$}

\uput[90](3.33,-1.0){$0$}
\uput[90](3.67,-1.0){$1$}
\uput[90](4.00,-1.0){$0$}
\uput[90](4.33,-1.0){$1$}
\uput[90](5.00,-1.0){$0$}
\uput[90](5.33,-1.0){$1$}
\uput[90](5.67,-1.0){$0$}
\uput[90](6.00,-1.0){$1$}
\uput[90](6.67,-1.0){$1$}
\uput[90](7.00,-1.0){$1$}

\end{pspicture}
\end{center}
\caption{Base Diagram and Construction for $a = 2$, $k = 3$ with $n = 10$}
\label{k=3}
\end{figure}

If we had instead chosen our $\mathcal K$ diagram to contain all horizontal lines, we would have instead constructed a rectangular ``thin strip,'' omnimosaic which consists of a list of all words of length $k$ repeated $k$ times.  Note however that the dimensions of this strip would be $k\times ka^k$.  In other words the ``areas" of the square and thin strip mosaics would be the same.  The thin strip omnimosaic, shown below for $a=k=2$,  is clearly the smallest possible array of dimension $k\times ka^k$ and provides the most intuitive construction of an omnimosaic:

$$\begin{pmatrix}
1&1\cr
1&0\cr
0&1\cr
0&0\cr
1&1\cr
1&0\cr
0&1\cr
0&0\cr
\end{pmatrix}
$$
It is also noteworthy that our strategy extends naturally to higher dimensional cases.  In this situation, rather then assigning to each unit square in a $k\times k$ square  a line parallel to one of the sides of the square, it is necessary to assign to each unit $d$-cube with a $d-1$ plane parallel to one of its sides.  By following a similar process, one can find a $d$-dimensional omnimosaic with side length of approximately $\displaystyle ka^{\frac{k^{d-1}}{d}}$.

\section{Threshold Behavior}  Isoperimetric considerations suggest that one ought to be able to produce smaller (in the sense of area) omnimosaics in square case than in the ``thin strip" case.  In this section, we use the probabilistic method to show that not only is this true, but that the pigeonhole bound is almost the best possible.  In other words, square omnimosaics can be constructed that almost yield a perfect covering of the $k\times k$ matrices.  

We use an elementary method (linearity of expectation and Markov's inequality) together with Suen's correlation inequality (see, e.g., \cite{janson}).  Let each entry of an $n\times n$ matrix be independently chosen to be one of the colors in the ``palette" ${\cal A}=\{1,2,\ldots,a\}$ with probability $1/a$.  Let $X$ be the number of ``missing" $k\times k$ matrices, i.e. matrices that cannot be found as a submatrix of the random $n\times n$ array.  Then
\[X=\sum_{j=1}^{a^{k^2}}I_j,\]
where $I_j=1$ (or $I_j=0$) according as the $j$th matrix is missing (or present) as a submatrix.  Our strategy can be summarized in a single line.  We will show that
\begin{eqnarray}\p({\rm array\ is\ not\ an\ omnimosaic})&=&\p(X\ge1)\nonumber\\ &\le& \e(X)\nonumber\\&=&\sum_{j=1}^{a^{k^2}}\p(I_j=1)<1\enspace({\rm or}\ \to0)\end{eqnarray}
if $n\ge\frac{ka^{k/2}}{e}(1+o(1))$, where we have used Markov's inequality and linearity of expectation, and where the last claim will be a consequence of Suen's inequality.

\medskip

\noindent{\bf Remark}  We remind the reader that given an independent sequence $X_1,X_2,\ldots$ of coin flips, the waiting time for a pure $k$-run of heads (or tails) is the largest of any waiting time among all $k$-patterns.  For example, the waiting time for $HHHHH$ is $2^1 + 2^2 + 2^3 + 2^4 +
2^5 +2^6 = 126$, but $HHTTHH$ occurs, on average, after just $2+4+64 = 70$
flips. The underlying reason for this is that a pure head run of length six
overlaps itself in six ways, but overlaps of $HHTTHH$ with itself can only
occur in one, two, or six places.  The situation is markedly different when we consider embedded subsequences rather than embedded strings:  In \cite{omni} it was shown that for a sequence $S$ on length $n$ on ${\cal A}$, the probability that a $k$-sequence is
missing as a subsequence of $S$ is {\it equal} to the probability that any other $k$-sequence is missing
in $S$.  It is our conjecture, however, that in the context of two dimensional mosaics, denoting by $J$ a matrix with monochromatic entries,
\[\p(J\ {\rm is\ missing})\ge\p(M\ {\rm is\ missing}),\]
where $M$ is a generic ${\cal A}$-valued matrix.  We were, however, unable to prove this fact, or even something weaker such as 
\[\p(J\ {\rm is\ missing})\ge\frac{1}{K}\p(M\ {\rm is\ missing})\] for some constant $K>1$; had we been able to, our strategy above would have been modified as follows:
\begin{eqnarray*}\p({\rm array\ is\ not\ an\ omnimosaic})&=&\p(X\ge1)\\
&\le& \e(X)\\
&\le&{a^{k^2}}\p(J\ {\rm is\ missing})<1\enspace({\rm or}\ \to0,)\end{eqnarray*}
where the convergence to zero would have been a consequence of the easier Janson exponential inequality (\cite{alon}).
\begin{thm}$$\omega(k,a)\le\frac{ka^{k/2}}{e}\lr1+o(1)\rr,$$ where we may take $o(1)=2\log k/k,\enspace k\to\infty.$\end{thm}
\begin{proof}For any $k\times k$ matrix $M$,
\[\p(M\ {\rm is\ missing})=\p\lr\bigcap_{j=1}^{{\nk}^2}B_j^C\rr=\p(Y=0),\]
where $A^C$ denotes the complement of $A$; $B_j$ occurs (equivalently $J_j=1$) if matrix $M$ is present in the $j$th of ${\nk}^2$ possible locations; and $Y=\sum_jJ_j$.  By Suen's inequality (the version in Theorem 2 of \cite{janson}), we have for any matrix $M$, 
$$\p(M\ {\rm is\ missing})\le\exp\lr-\mu+\Delta_M e^{2\delta}\rr,\eqno(*)$$
where
\[\mu=\e(Y)={\nk}^2\frac{1}{a^{k^2}},\]
$i\sim j$ if the $i$th and $j$th locations share at least one position (and thus at least one row and column),
\[\Delta_M=\sum_{\{i,j\}:i\sim j}\e(I_iI_j),\]
\[\delta_i=\sum_{j\sim i}\p(I_j=1),\]
and
\[\delta=\max_i\delta_i.\]
The computation of $\Delta_M$ is the first component of the proof, and follows the development in \cite{heidi} and \cite{ericben} (where the focus was, given $n$ and $k$, to find a threshold $p$, and where the ``continuous time" analysis was done using gamma functions).  
\begin{eqnarray}
\Delta_M&=&\sum_{\{i,j\}:i\sim j}\e(I_iI_j)\nonumber\\
&\le&\frac{{\nk}^2}{a^{2k^2}}\sum_{{r,c=1}\atop{r+c<2k}}^k{k\choose r}{k\choose c} {n\choose{k-r}}{n\choose {k-c}}a^{rc}\nonumber\\
&=&\frac{{\nk}^2}{a^{2k^2}}\sum_{{r,c=1}\atop{r+c<2k}}^k\varphi(r,c)\nonumber\\
&\le&\frac{{\nk}^2}{a^{2k^2}}k^2\max\{\varphi(r,c):1\le r,c\le k; r+c<2k\},
\end{eqnarray}
where $\varphi(r,c)={k\choose r}{k\choose c} {n\choose{k-r}}{n\choose {k-c}}a^{rc}$.  The analysis of $\varphi(r,c)$ is the content of the next few lemmas.  Note that in (4) above, we have upper bounded $\Delta_M$ by $\Delta_J$.
\begin{lem}
Given $c$, $1\le c\le k$, $\varphi(.,c)$ is either monotone or unimodal as a function of $r$.
\end{lem}
\begin{proof}
The function $\varphi(.,c)$ is increasing if and only if 
\[\frac{\varphi(r+1,c)}{\varphi(r,c)}=\frac{(k-r)^2}{(r+1)(n-k+r+1)}a^c\ge1,\]
or if
\begin{equation}\frac{(k-r)^2}{(r+1)(n-k+r+1)}\ge\frac{1}{a^c}.\end{equation}
But the left side of (5) is a decreasing function of $r$.  This completes the proof.
\end{proof}
\begin{lem} $\varphi(1,1)\ge\varphi(2,1)$ if $n\ge \frac{k^2a}{2}+k-2.$\end{lem}
\begin{proof} This follows easily from the definition of $\varphi$.\end{proof}
\begin{lem} $\varphi(k,k)\ge\varphi(k-1,k)$, if $n\le a^k/k$ and thus $\varphi(k-1,k)=\max\{\varphi(r,k):1\le r\le k-1\}$ under this condition. \end{lem}
\begin{proof}Again, this follows easily from the definition of $\varphi$.\end{proof}
\begin{lem} $\varphi(k-1,k)\ge\varphi(1,1)$ provided that $k$ is large enough.
\end{lem}
\begin{proof} The condition $\varphi(k-1,k)\ge\varphi(1,1)$ is equivalent to 
\[nka^{k(k-1)}\ge k^2{{n}\choose{k-1}}^2a,\]
and is thus satisfied if 
\begin{equation}nka^{k(k-1)}\ge k^2\lr\frac{ne}{k-1}\rr^{2k-2}a.\end{equation}
Setting $n=\frac{ka^{k/2}}{e}(1+o^*(1))$ where $o^*(1)=\frac{B\log k+C\log\log a+o(1)}{k}$, as it will end up being by the end of the proof of the theorem, we see that (6) holds if
\[a^{k^2-k/2}\ge2e^3a^{k^2-k+1}(1+o^*(1))^{2k-2},\]
i.e., if
\begin{equation}a^{k/2-1}\ge2e^3\lr1+\frac{B\log k+C\log\log a+o(1)}{k}\rr^{2k-2}.\end{equation}
Since \[\lr1+\frac{B\log k+C\log\log a+o(1)}{k}\rr^{2k-2}\le \exp\{2B\log k+2C\log\log a+2o(1)\}(1+o^{**}(1)),\] we see that (7) holds if \[a^{k/2-1}\ge 3e^3(\log a)^{2C}k^{2B},\]
i.e., if $k$ is large enough.
\end{proof}
The maximum of $\ph$ occurs either on the boundary of its domain or around its discretized critical points, i.e. points $(r,c)$ closest to satisfying $\ph(r+1,c)=\ph(r,c)$ and $\ph(r,c+1)=\ph(r,c)$. Now $\ph(r+1,c)=\ph(r,c)$ if
\[\frac{(k-r)^2a^c}{(r+1)(n-k+r+1)}=1,\]
while $\ph(r,c+1)=\ph(r,c)$ if
\[\frac{(k-c)^2a^r}{(c+1)(n-k+c+1)}=1,\]
Thus $(r,c)$ is a critical point only if
\[\frac{(k-r)^2}{(r+1)(n-k+r+1)}\frac{1}{a^r}=\frac{(k-c)^2}{(c+1)(n-k+c+1)}\frac{1}{a^c}.\]
But both $(k-x)^2/(x+1)(n-k+x+1)$ and $1/a^x$ are non-negative decreasing functions and thus 
\[\frac{(k-r)^2}{(r+1)(n-k+r+1)}\frac{1}{a^r}=\frac{(k-c)^2}{(c+1)(n-k+c+1)}\frac{1}{a^c}\Rightarrow r=c.\]  This shows that the maximum of $\ph$ will occur either on the boundary of its domain or around $(r,r)$ for some $r$.
\begin{lem} $\ph(r,r)$ is first decreasing and then increasing as a function of $r$.
\end{lem}
\begin{proof} We show that $\sqrt{\ph(r,r)}=\gamma(r)$ satisfies the required condition.  Notice that $\gamma$ is increasing when
\[\frac{\gamma(r+1)}{\gamma(r)}=\frac{a^{r+(1/2)}(k-r)^2}{(r+1)(n-k+r+1)}\ge 1,\]
or if
\begin{equation}\rho(r):=2\log(k-r)-\log(r+1)-\log(n-k+r+1)\ge-\lr r+\frac{1}{2}\rr\log a.\end{equation}
Now both the left and right sides of (8) are decreasing functions of $r$, but more is true: The right side of (8) decreases linearly, while the left side, $\rho$, has negative third derivative, and hence its concavity decreases.  Thus $\rho$ is either convex, concave, or first convex and then concave.  Now for large $k$, (8) does not hold for $r=1$ since the condition $n\le ak^2$ is not true, but (8) does hold for $r=k-1$ since we may assume that $n\le a^k/k$.  Now regardless of which of the three convexity scenarios
is actually valid for $\rho(r)$, it follows that $\gamma$ is decreasing for
 $r\le r_0$ and increasing thereafter, as asserted.\end{proof}

Now, to show that $\ph(k-1,k)$ is larger than all the values of $\ph$ 
 in the interior of $D$, we need to compare $\ph(k-1,k)$ to the critical points of $\ph$, and thus, by Lemmas 3.2 to 3.6, only 
to the values of 
$\ph(1,1)$ and $\ph(k-1,k-1)$.  We have already seen that $\ph(k-1,k)\ge\ph(1,1)$, and it is straightforward that $\ph(k-1,k)\ge\ph(k-1,k-1)$ if $n\le a^{k-1}/k$.  The above lemmas thus show that $\ph(k-1,k)=\max\{\ph(r,c):r,c=1,2,\ldots,k;r+c<2k\}$.  We may thus bound $\Delta$ as follows:
\begin{eqnarray}
\Delta&\le&\frac{{\nk}^2}{a^{2k^2}}k^2\max\{\varphi(r,c):1\le r,c\le k; r+c<2k\}\nonumber\\
&=&\frac{{\nk}^2}{a^{2k^2}}k^2\ph(k-1,k)\nonumber\\
&=&\frac{{\nk}^2}{a^{2k^2}}nk^3a^{k(k-1)}\nonumber\\
&=&\frac{nk^3}{a^k}\mu.
\end{eqnarray}
Turning our attention to $\delta$, we see that
\begin{eqnarray}\delta&=&\max_i\sum_{j\sim i}\p(I_j=1)\nonumber\\
&\le&k^2{{n}\choose{k-1}}^2\frac{1}{a^{k^2}}\nonumber\\
&=&\frac{k^4}{(n-k+1)^2}\mu\nonumber\\
&\le&\frac{2k^4}{n^2}\mu.
\end{eqnarray}
Substituting (9) and (10) into (*), we see that for any matrix $M$
\[\p(M\ {\rm is\ missing})\le\exp\lr-\mu+\mu\frac{nk^3}{a^k}e^{4k^4\mu/n^2}\rr,\]
so that by (3), 
\begin{equation}\p({\rm array\ is\ not\ an\ omnimosaic})\le a^{k^2}\exp\lr-\mu+\mu\frac{nk^3}{a^k}e^{4k^4\mu/n^2}\rr.\end{equation}
Assuming that 
\[\frac{ka^{k/2}}{e}\le n\le \frac{2ka^{k/2}}{e},\]
we see that (11) yields
\begin{equation}\p({\rm array\ is\ not\ an\ omnimosaic})\le a^{k^2}\exp\lr-\mu+\mu\frac{k^4}{a^{k/2}}e^{36k^2\mu/a^k}\rr.\end{equation}
Now if $n$ is further restricted so that 
\[n\le \frac{ka^{k/2}}{e}\lr1+\frac{2\log k}{k}\rr,\]
we see that
\[\mu\le\lr\frac{ne}{k}\rr^{2k}\frac{1}{a^{k^2}}\le k^4,\]
so that (12) yields
\begin{eqnarray}\p({\rm array\ is\ not\ an\ omnimosaic})&\le& a^{k^2}\exp\lr-\mu+\frac{k^8}{a^{k/2}}e^{36k^6/a^k}\rr\nonumber\\&=&a^{k^2}\exp\lr-\mu+o^{**}(1)\rr\end{eqnarray}
We next estimate as follows, using the first order upper bound in Stirling's formula, as, e.g., in \cite{feller}:
\begin{eqnarray}e^{-\mu}&=&\exp\lc -{n\choose k}^2\frac{1}{a^{k^2}}\rc\nonumber\\
&\le&\exp\lc-\frac{(n-k)^{2k}}{k!^2}\frac{1}{a^{k^2}}\rc\nonumber\\
&\le&\exp\lc-\frac{(n-k)^{2k}e^{2k}}{(2\pi k) k^{2k}}\frac{1}{\lr1+\frac{1}{12k}\rr^2}\frac{1}{a^{k^2}}\rc.
\end{eqnarray}
We now set 
\begin{equation}n-k=\frac{ka^{k/2}}{e}(2\pi k)^{1/2k}\lr1+\frac{1}{12k}\rr^{1/k}\exp\lc\frac{\log k}{k}+\frac{\log\log a}{2k}+\frac{o(1)}{k}\rc\end{equation} in (13) and (14) to get
\begin{eqnarray}\p({\rm array\ is\ not\ an\ omnimosaic})&\le& a^{k^2}\exp\lr-e^{2\log k+\log\log a+2o(1)}\rr e^{o^{**}(1)}\nonumber\\
&=&a^{k^2}\frac{1}{a^{k^2(1+{\tilde {o}}(1)})}e^{o^{**}(1)}<1\nonumber\\
\end{eqnarray} for suitably chosen $o(1)$ and ${\tilde{o}}(1)$ functions that do not go to zero too rapidly.  Notice that the value of $n$ given by (15) is no larger than that announced in the statement of the theorem.\end{proof}
\section{Open Questions}
The first of questions mentioned below is already under investigation:  We need  

(i) to work out the details of results analogous to Theorem 3.1 for graphs, weighted graphs, hypergraphs, non-rectangular arrays, and omnimosaics in higher dimensions; 

(ii) to improve the upper bound in Theorem 3.1.

(iii) to understand why there is an intrinsic difference between one and two-dimensional omni behavior (recall that in the former case there was a ``gap" between omni thresholds and the blow-up threshold for the expected number of missing $k$-subsequences);

(iv) to produce a combinatorial argument or bijection that leads to a proof of our conjecture that $\p(M\ {\rm is\ missing})\le B_k\p(J\ {\rm is\ missing})$.  Our conjecture is based on robust numerical evidence and several ``near-proofs," and it is important to note that while we believe that $B_k=1$, our agenda would be realized even if $B_k$ were a huge constant, or a polynomial, or even as rapidly growing as $a^{\sqrt k}$; and 

(v) to understand the role of monotonicity and extend our agenda to excluded permutation matrices, as studied, e.g., by \cite{marcustardos}.
 
\section{Acknowledgements} The research of all three authors was supported by NSF Grant 1004624, and conducted by KRB and NGT as part of their REU Project.

\end{document}